\documentclass[final]{siamart190516}

\usepackage{dsfont}
\usepackage{lipsum}
\usepackage{amsfonts}
\usepackage{graphicx}
\usepackage{epstopdf}
\usepackage{algorithmic}
\usepackage{comment}

\ifpdf
\DeclareGraphicsExtensions{.eps,.pdf,.png,.jpg}
\else
\DeclareGraphicsExtensions{.eps}
\fi

\newsiamremark{remark}{Remark}
\newsiamremark{hypothesis}{Hypothesis}
\crefname{hypothesis}{Hypothesis}{Hypotheses}
\newsiamthm{claim}{Claim}

\headers{Moments for Vlasov-Poisson with magnetic field}{A. Rege}

\title{The Vlasov-Poisson system with a uniform magnetic field: propagation of moments and regularity\thanks{Submitted to the editors 26th October 2020.}}

\author{Alexandre Rege\thanks{Sorbonne-Universit\'e, CNRS, Universit\'e de Paris, Laboratoire Jacques-Louis Lions (LJLL), 4 place Jussieu, F-75005 Paris, France 
		(\email{alexandre.rege@sorbonne-universite.fr}, \url{https://arege.pages.math.cnrs.fr/}).}}

\usepackage{amsopn}

\newcommand{\norme}[1]{\left\Vert #1\right\Vert}
\newcommand{\abs}[1]{\left\lvert #1\right\rvert}
\newcommand{\R}{\mathbb{R}}
\newcommand{\NN}{\mathbb{N}}
\newcommand{\intervalleoo}[2]{\mathopen{]}#1\, ,#2\mathclose{[}}
\newcommand{\intervalleff}[2]{\mathopen{[}#1\, ,#2\mathclose{]}}
\newcommand{\intervalleof}[2]{\mathopen{]}#1\, ,#2\mathclose{]}}
\newcommand{\intervallefo}[2]{\mathopen{[}#1\, ,#2\mathclose{[}}

\begin{document}

\maketitle

\begin{abstract}
  We show propagation of moments in velocity for the 3-dimensional Vlasov-Poisson system with a uniform magnetic field $B=(0,0,\omega)$ by adapting the work of Lions, Perthame. The added magnetic field also produces singularities at times which are the multiples of the cyclotron period $t=\frac{2\pi k}{\omega}$, $k \in \NN$. This result also allows to show propagation of regularity for the solution.
  For uniqueness, we extend Loeper's result by showing that the set of solutions with bounded macroscopic density is a uniqueness class.

\end{abstract}



\section{Introduction}
We consider the Cauchy problem for the Vlasov-Poisson system with an external magnetic field, which is given by
\begin{equation}\label{sys:VPwB}
\left\{
\begin{aligned}
& \partial_t f + v\cdot \nabla_x f + E \cdot \nabla_{v} f + v \wedge B \cdot \nabla_{v}	 f= 0, \\
& E(t,x)= \frac{1}{4\pi}\int_{\R^3}\frac{x-y}{\abs{x-y}^3}\rho(t,y)dy \text{ with }\rho(t,x):=\int_{\mathbb R^3} f(t,x) dv\\
& f(0,x,v)=f^{in}(x,v)\geq 0.
\end{aligned}
\right.
\end{equation}

This set of equations governs the evolution of a cloud of charged particles, where $f(t,x,v)$ is the distribution function at time $t \geq 0$, position $x \in \mathbb{R}^3$ and velocity $v \in \mathbb{R}^3$. $E$ corresponds to the self-consistent electric field and $B$ is an external, constant and uniform magnetic field given by 

\begin{equation}\label{exp:magn}
B=\begin{pmatrix}
0\\
0 \\
\omega 
\end{pmatrix}
\end{equation}
where $\omega>0$ is the cyclotron frequency.

The unmagnetized Vlasov-Poisson system has been extensively studied with the works of Arsenev \cite{AR75} for weak solutions, Okabe and Ukai in dimension 2 \cite{OU78} and Bardos and Degond for small initial data \cite{BD85}. In the case of general initial data in dimension 3, two main approaches have been developed. The first one is based on the study of the charateristic curves with the papers from Pfaffelmoser and Sch\"affer \cite{P92,S91}. The second approach, first introduced for Vlasov type equations by Lions and Perthame \cite{LP91}, is based on the propagation of moments of the distribution function. This has resulted in several works where similar propagation properties are shown in the case of more general systems \cite{GJP00} and also in the case of more general assumptions \cite{C99,CZ16,P12,P14}.

As for the Vlasov-Poisson system with an external magnetic field, it is a system of considerable importance for the modeling of tokamak plasmas. For this reason, there exists an abundant literature on the case with strong magnetic field, where the aim is to derive asymptotic models \cite{B09,DF16,FS98,FS00,GS99,GS03} and devise numerical methods that capture this asymptotic behavior \cite{CLMZ17,FHLS15}. The Vlasov-Poisson system with an external and homogeneous magnetic field has also been studied in the half-space and in an infinite cylinder in \cite{S14,ST17}.
	
With the external magnetic field, the first difficulty is finding an appropriate representation formula for the macroscopic density, since the characteristics are a lot more complex than in the case without magnetic field. The second and most arduous difficulty is the existence of singularities at times $t=0,\frac{2\pi}{\omega}, \frac{4\pi}{\omega},...$, which correspond to the cyclotron periods, when we try to control the electric field. We manage to avoid these singularities because our estimates are valid for $t \in \intervalleff{0}{T_\omega}$ with $T_
\omega=\frac{\pi}{\omega}$ which is independent of $f^{in}$. This allows us to reiterate our analysis on $\intervalleff{T_\omega}{2 T_\omega}$ and so on.

Hence, in this paper, we succeed in extending the results of \cite{LP91} to the case of Vlasov-Poisson with an homogeneous external magnetic field. This is a first step to proving propagation of moments in the case of a non-homogeneous magnetic field.

First, we detail our main result and several additional results in \cref{sec:res}. Then, in \cref{sec:prel}, we continue by presenting the basic definitions and lemmas that will be necessary for the proof of our main result in \cref{sec:proofmain}, which is the core of this work. More precisely, we will give the new representation formula for the macroscopic density in \cref{ssec:reprho} and show how we control the electric field with the "magnetized" characteristics in \cref{ssec:controlelec}. To treat the singularities that appear, we establish a Gr\"onwall inequality on $\intervalleff{0}{T_\omega}$ in \cref{ssec:gronwall} and show how this leads to propagation of moments for all time in \cref{ssec:propaallt}. In \cref{ssec:Bsource}, we explore a method where we place the magnetic part of the Lorentz force in the source term, which doesn't work, but is so simple that it's still interesting to mention. Finally, we will give the proofs of our additional results in \cref{sec:proofadditional}. In particular, we will explicit a new condition on the initial data so as to obtain the boundedness of the macroscopic density.

\section{Results}\label{sec:res}
First we give some notations and definitions.

For $k \geq 0$ we denote the $k$-th order moment density and the $k$-th
order moment in velocity of a non-negative, measurable function $f \colon \R^6 \rightarrow \intervallefo{0}{\infty}$
by
\begin{equation*}
m_k(f)(x):=\int \abs{v}^k f dv \quad \text{and} \quad M_k(f):=\int m_k(f)(x)dx=\iint \abs{v}^k f dvdx.
\end{equation*}

We write $\mathcal{E}(t)$ for the energy of system \cref{sys:VPwB}, which is given by
\begin{equation}\label{eq:energy}
\mathcal{E}(t):=\frac{1}{2}\iint_{\R^3\times \R^3}\abs{v}^2f(t,x,v)dxdv+\frac{1}{2}\int_{\R^3}\abs{E(t,x)}^2dx,
\end{equation}
and we also write $\mathcal{E}_{in}:=\mathcal{E}(0)$.

Lastly we define the notion of solutions for the magnetized Vlasov-Poisson system \ref{sys:VPwB}, which is analogous the notion of weak solutions used in Arsenev \cite{AR75}.
\begin{definition}
Let $f \in L^\infty(\R_+;L^1(\R^3 \times \R^3) \cap L^2(\R^3 \times \R^3))$, $f$ is a weak solution of the magnetized Vlasov-Poisson system if $f\abs{v}^2 \in L^\infty(\R_+;L^1(\R^3 \times \R^3))$ and we have
\begin{equation}
\partial_t f+\text{\rm div}_x(vf)+\text{\rm div}_v((E+v \wedge B)f)=0 \text{ in } \mathcal{D}'(\R_+ \times \R^3 \times \R^3).
\end{equation}
\end{definition}
\subsection{Main result}\label{ssec:main}
First we present this paper's main result: propagation of velocity moments for the Vlasov-Poisson system with an external magnetic field.
\begin{theorem}[Propagation of moments]\label{theo:main}
Let $k_0>3, T>0, f^{in}=f^{in}(x,v)\geq 0$ a.e. with $f^{in}\in L^1\cap L^\infty (\R^3\times \R^3)$ and assume that
\begin{equation}\label{ineq:momentini}
\iint_{\R^3\times \R^3}\abs{v}^{k_0}f^{in}dxdv < \infty.
\end{equation}
Then for all $k$ such that $0\leq k \leq k_0$, there exists\\ $C=C(T,k,\omega,\norme{f^{in}}_1,\norme{f^{in}}_\infty,\mathcal{E}_{in},M_k(f^{in})) > 0$ and a weak solution 
\begin{equation}
f \in C(\R_+;L^p(\R^3 \times \R^3)) \cap L^\infty(\R_+;L^p(\R^3 \times \R^3))
\end{equation} 
$(1\leq p < +\infty)$ to the Cauchy problem for the Vlasov-Poisson system with magnetic field \cref{sys:VPwB} with $B$ given by \cref{exp:magn}
in $\R^3 \times \R^3$ such that
\begin{equation}\label{ineq:moment}
\iint_{\R^3\times \R^3}\abs{v}^{k}f(t,x,v)dxdv \leq C < +\infty,\quad 0\leq t \leq T.
\end{equation}
\end{theorem}
\begin{remark}
As said in \cite{LP91}, the assumptions in \cref{theo:main} guarantee that the initial energy $\mathcal{E}_{in}$ is finite.
\end{remark}
\begin{remark}
Like in the original paper, all the apriori estimates that will be presented in the proof are true for smooth solutions ($C^\infty$ with compact support). Since our estimates depend only on $T,k,\omega,\norme{f^{in}}_1,\norme{f^{in}}_\infty,\mathcal{E}_{in},M_k(f^{in})$, it is sufficient to pass to the limit in the approximate magnetized Vlasov-Poisson system. This approximate system is obtained by applying the regularization procedure introduced by Arsenev \cite{AR75} to prove the existence of weak solutions.
\end{remark}
Let's first mention that to prove the existence of weak solutions to \cref{sys:VPwB} is relatively straightforward by adapting Arsenev's work \cite{AR75}, even when the external magnetic field isn't homogeneous. The only requirement is to have $B \in L^\infty(\R^3)$.

As said above, \cref{theo:main} is an extension of the main result in \cite{LP91}. To obtain \cref{ineq:moment}, we follow approximately the same strategy, which is to establish a linear Gr\"onwall inequality on the velocity moment. First, by writing a differential inequality on the velocity moment, we realize that to obtain a Gr\"onwall inequality on the moments, we need to control a certain norm of the electric field. To do this, we require the information gained from the Vlasov equation. Hence, by using the characteristics, we can express the macroscopic charge density with a representation formula, which will in turn allow us to control the norm of the electric field. In our case, the added magnetic field significantly complicates the characteristics and the initial proof by extension. 

\subsection{Additional results}\label{ssec:additional}

Now we state a result regarding propagation of regularity for solutions to \ref{sys:VPwB}, where the initial condition is sufficiently regular. This is also an extension of a result stated in \cite{LP91} to the case with magnetic field. However here we present this result and its proof with much more detail than in \cite{LP91} by adapting section 4.5 of \cite{G13}.
\begin{theorem}[Propagation of regularity]\label{theo:reg} Let $h \in C^1(\R)$ such that 
	\begin{equation*}
		h \geq 0, h'\leq 0 \text{ and } h(r)=\mathcal{O}(r^{-\alpha}) \text{ with } \alpha >3.
	\end{equation*}
	and let $f^{in} \in C^1(\R^3)$ a probability density on $\R^3\times \R^3$ such that $f^{in}(x,v)\leq h(\abs{v})$ for all $x,v$ and which verifies
	\begin{equation}
	\iint_{\R^3\times \R^3}(1+\abs{v}^{k_0})f^{in}(x,v)dxdv < \infty
	\end{equation}
	with $k_0>6$.\\
	Then there exists a weak solution of the Cauchy problem for the Vlasov-Poisson system with magnetic field \cref{sys:VPwB} $(f,E)\in C^1(\R_+\times \R^3\times \R^3)\times C^1(\R_+\times \R^3)$ satisfying the decay estimate
	\begin{equation}\label{eq:decay}
	\underset{(t,x) \in \intervalleff{0}{T}\times \R^3}{\sup} f(t,x,v)+\abs{D_x f(t,x,v)}+\abs{D_v f(t,x,v)} = \mathcal{O}(\abs{v}^{-\alpha})
	\end{equation}
	for all $T>0$.
\end{theorem}

Next, we state a result on the uniqueness of solutions to \ref{sys:VPwB} which is a direct adaptation of Loeper's paper \cite{L06}.

\begin{theorem}[Uniqueness]\label{theo:uniq}
Let $f^{in} \in L^1\cap L^\infty (\R^3\times \R^3)$ be a probability density such that for all $T>0$
\begin{equation}
\norme{\rho}_{L^\infty(\intervalleff{0}{T} \times \R^3)} < +\infty
\end{equation}
then there exists at most one solution to the Cauchy problem for the Vlasov-Poisson system with magnetic field \cref{sys:VPwB}.	
\end{theorem}
Finally, we give a proposition that allows to build solutions with bounded macroscopic density, which is analogous to the condition given in Corollary 3 in \cite{LP91}.
\begin{proposition}\label{prop:boundrho}
Let $f^{in}$ verify the assumptions of \cref{theo:main} with $k_0 > 6$ and assume that $f^{in}$ also satisfies
\begin{equation}\label{eq:boundrho}
\begin{aligned}
 \text{ess sup}\{f^{in}(y+vt,w), \abs{y-x}\leq (R+\omega \abs{v})t^2e^{\omega t},\abs{w-v}\leq (R+\omega \abs{v})te^{\omega t} \}\\
 \in L^\infty(\intervallefo{0}{T}\times \R^3_x,L^1(\R^3_v))
\end{aligned}
\end{equation}
for all $R>0$ and $T>0$.

Then, the solution of \cref{sys:VPwB} verifies
\begin{equation}
\rho \in L^\infty(\intervalleff{0}{T} \times \R^3_x)
\end{equation}
for all $T>0$.
\end{proposition}
\section{Preliminaries}\label{sec:prel}
As said above, we now present some basic results necessary for the proofs.
First we recall the weak Young inequality. The proof of this basic inequality can be found in \cite{LL97}.
\begin{lemma}[Weak Young inequality]
	\label{w_Young}
	Let $1<p,q,r<\infty$ with $\frac{1}{p}+\frac{1}{q}=1+\frac{1}{r}$, then for all functions $f \in L^p(\R^n)$, $g \in L^{q,w}(\R^n)$ the convolution product $f \star g= \int_{\R^n} f(y)g(\cdot - y)dy \in L^r(\R^n)$ and satisfies 
	
	\begin{equation}
	\label{w_Youn}
	\norme{f\star g}_r\leq c\norme{f}_p \norme{g}_{q,w}
	\end{equation} 
	with $c=c(p,q,n)$ and by definition $g \in L^{q,w}(\R^n)$ iff $h$ is measurable and
	\begin{equation}
	\underset{\tau>0}{\sup}\left( \tau \left(\text{\rm vol}\left\{x \in \R^n \mid \abs{g(x)}>\tau \right\}\right)^\frac{1}{q}\right)<\infty.
	\end{equation}
	Furthermore, we can define a norm on $L^{q,w}(\R^n)$ given by
	\begin{equation}
     \norme{f}_{q,w}=\underset{\abs{A}< \infty}{\sup} \abs{A}^{-\frac{1}{q'}}\int_A \abs{f(x)}dx.
	\end{equation}
\end{lemma}
The next three lemmas and their proofs can be found in \cite{R05}. It is easy to show that the estimates given in \cref{lem:estE} are also true in our case. \cref{lem:ineq_mk} is a fundamental velocity moment inequality  and \cref{lem:lineq_1} is a basic functional inequality.
\begin{lemma}\label{lem:estE}
	The estimate
	
	\begin{equation}\label{est_E}
	\norme{E(t)}_p \leq C, t \in \intervallefo{0}{T}
	\end{equation}
	holds for $p\in \intervalleof{\frac{3}{2}}{\frac{15}{4}}$ with the constant $C=C(\norme{f^{in}}_1,\norme{f^{in}}_\infty,\mathcal{E}_{in})$ independent of $p$, so that we also have the estimate
	\begin{equation}
	\norme{E(t)}_{\frac{3}{2},w}\leq C, t \in \intervallefo{0}{T}.
	\end{equation}
\end{lemma}

\begin{lemma}\label{lem:ineq_mk}	
	Let $1 \leq p,q \leq \infty$ with $\frac{1}{p}+\frac{1}{q}=1$, $0\leq k'\leq k < \infty$ and $r=\frac{k+\frac{3}{q}}{k'+\frac{3}{q}+\frac{k-k'}{p}}$. 
	If $f\in L^p_+(\R^6)$ with $M_k(f) < \infty$ then $m_{k'}(f)\in L^r(\R^3)$ and 
	
	\begin{equation}
	\norme{m_{k'}(f)}_r \leq c \norme{f}_p^{\frac{k-k'}{k+\frac{3}{q}}}M_k(f)^{\frac{k'+\frac{3}{q}}{k+\frac{3}{q}}}
	\end{equation} where $c=c(k,k',p)>0$.
\end{lemma}
\begin{lemma}\label{lem:lineq_1}
	For all functions $g \in L^1\cap L^\infty(\R^3)$ and $h \in L^{\frac{3}{2},w}(\R^3)$,
	\begin{equation}\label{ineq_1}
	\int_{\R^3} \abs{gh}dx \leq 3\left(\frac{3}{2}\right)^\frac{2}{3}\norme{g}_1^\frac{1}{3}\norme{g}_\infty^\frac{2}{3}\norme{h}_{\frac{3}{2},w}
	\end{equation}
\end{lemma}
Lastly we give a Calder\'on-Zygmund inequality, whose proof one can find in \cite{D01}. 
\begin{lemma}[Calder\'on-Zygmund]
	\label{cal_zy}
	
	If $\Omega \in L^q(\mathcal{S}^{d-1})$, $q>1$ so that $\int_{\mathcal{S}^{d-1}} \Omega(\omega) dS(\omega)=0$, we consider the tempered distribution $T= \text{\rm vp} \frac{\Omega\left(\frac{x}{\abs{x}}\right)}{\abs{x}^d} \in \mathcal{S}'(\R^d)$. The operator $\phi \in \mathcal{D}(\R^d) \mapsto T\star \phi$ can be uniquely extended into a bounded operator on $L^p(\R^d)$ for $p \in \intervalleoo{1}{\infty}$.
\end{lemma}
\section{Proof of propagation of moments}\label{sec:proofmain}
As said above, we extend the main result of \cite{LP91} to the case of Vlasov-Poisson with a homogeneous magnetic field. However, here, we use the same steps for the proof as in \cite{R05}, where the ideas of \cite{LP91} are presented.

We begin by considering $k_0,T$ and $f^{in}$ that follow the assumptions of \cref{theo:main}. Then, as in \cite{R05}, we can write a differential inequality on $M_k$, with $0\leq k \leq k_0$.

We differentiate $M_k$, and by integration by parts, a H\"older inequality and lemma \eqref{lem:ineq_mk} with $p=\infty, q=1, k'=k-1$, we obtain,

\begin{align*}
\abs{\frac{d}{dt} M_k(t)} & =	\abs{\iint \abs{v}^k (-v\cdot\nabla_{x}f-(E+v\wedge B)\cdot \nabla_{v} f) dv dx}\\
& = \abs{\iint \abs{v}^k \mathop{\rm div_v}\left((E+v\wedge B)f\right)dvdx}\\
& = \abs{\iint k\abs{v}^{k-2} v\cdot E f dv dx}\\
& \leq \iint k\abs{v}^{k-1} f dv \abs{E} dx\\
& \leq k \norme{E(t)}_{k+3}\norme{m_{k-1}(f)}_{\frac{k+3}{k+2}}
\end{align*}
and finally
\begin{equation}\label{eq:diff_Mk}
\abs{\frac{d}{dt} M_k(t)}\leq C \norme{E(t)}_{k+3}M_k(t)^{\frac{k+2}{k+3}}
\end{equation}
with $C=c(k)\norme{f(t)}_\infty^\frac{1}{k+3}=C(k,\norme{f^{in}}_\infty)$. The computations above are almost the same as in the original case because the magnetic part vanishes. This means that, like in the unmagnetized case, we need to control $\norme{E(t)}_{k+3}$ to obtain a Gr\"onwall inequality on $M_k$.
\subsection{A representation formula for $\rho$}\label{ssec:reprho}
Now we turn to the next step of the proof. Following \cite{R05}, we write a representation formula for the macroscopic density using the characteristics associated to the Vlasov equation. With the added magnetic field, the characteristics are much more complicated than in the unmagnetized case. This translates to a generalized representation formula for the macroscopic density.
\begin{lemma}\label{lem_rep_rho}
We have the following representation formula for $\rho$, 
\begin{equation}\label{rep_rho}
\rho(t,x)=\underset{=:\rho_0(t,x)}{\underbrace{\int_v f^{in}(X^0(t),V^0(t))dv}}+ \text{\rm div}_x \int_0^t\int_v \left(fH_t\right)\left(s,X(s;t,x,v),V(s;t,x,v)\right)dvds
\end{equation}
with $\left(X(s;t,x,v),V(s;t,x,v)\right)$ the characteristics associated to the Vlasov equation of system \cref{sys:VPwB}, given by
\begin{equation}\label{chara}
\left\{
\begin{aligned}
& V(s;t,x,v)=  \begin{pmatrix}
v_1\cos(\omega (s-t))+v_2\sin(\omega (s-t))\\
-v_1\sin(\omega (s-t))+v_2\cos(\omega (s-t))\\
v_3
\end{pmatrix}\\
& X(s;t,x,v)=\begin{pmatrix}
x_1+\frac{v_1}{\omega}\sin(\omega (s-t))+\frac{v_2}{\omega}(1-\cos(\omega (s-t)))\\
x_2+\frac{v_1}{\omega}(\cos(\omega (s-t))-1)+\frac{v_2}{\omega}\sin(\omega (s-t))\\
x_3+v_3(s-t))
\end{pmatrix}
\end{aligned}
\right.
\end{equation}
with $(X^0(t),V^0(t))=(X(0;t,x,v),V(0;t,x,v))$ and 
\begin{equation}\label{express_H}
H_t\left(s,x\right)=
\begin{pmatrix}
\frac{\sin(\omega (s-t))}{\omega} E_1(s,x)+\frac{\cos(\omega (s-t))-1}{\omega}E_2(s,x)\\
\frac{1-\cos(\omega (s-t))}{\omega} E_1(s,x)+\frac{\sin(\omega (s-t))}{\omega}E_2(s,x)\\
(s-t)E_3(s,x)
\end{pmatrix}
\end{equation}
with $E_i$ the coordinates of the electric field $E$.
\end{lemma}
\begin{proof}
Firstly, thanks to the Vlasov equation, which we see as a transport equation in $x$ and $v$ with source term $-E\cdot\partial_v f$, we can express $f$ by solving the characteristics and by applying the Duhamel formula
\begin{equation*}
f(t,x,v)=f^{in}(X^0(t),V^0(t))-\int_0^t \text{\rm div}_v(fE)(s,X(s;t,x,v),V(s;t,x,v))ds
\end{equation*}
where $\left(X(\cdot,t,x,v),V(\cdot,t,x,v)\right)$ is the solution to
\begin{equation*}
\left\{
\begin{aligned}
& \frac{d}{ds}\left(X(s;t,x,v),V(s;t,x,v)\right)=\left(V(s;t,x,v),\omega V_2(s;t,x,v),-\omega V_1(s;t,x,v),0 \right)\\
& \left(X(t;t,x,v),V(t;t,x,v)\right)=(x,v),
\end{aligned}
\right.
\end{equation*}
hence the expressions in \eqref{chara}. Now if we consider
\begin{equation*}
G_t(s,x)=
\begin{pmatrix}
\cos(\omega (s-t))E_1(s,x)-\sin(\omega (s-t))E_2(s,x)\\
\sin(\omega (s-t))E_1(s,x)+\cos(\omega (s-t))E_2(s,x)\\
E_3(s,x)
\end{pmatrix}
\end{equation*}
then
\begin{align*}
& \text{\rm div}_v\int_0^t fG_t(s,X(s;t,x,v),V(s;t,x,v)) ds\\
& =\int_0^t \cos(\omega(s-t)) \partial_{v_1}\left(fE_1\left(s,X(s;t,x,v),V(s;t,x,v)\right)\right)\\
& -\int_0^t\sin(\omega(s-t)) \partial_{v_1}\left(fE_2\left(s,X(s;t,x,v),V(s;t,x,v)\right)\right)\\ & +\int_0^t\sin(\omega(s-t))\partial_{v_2}\left(fE_1\left(s,X(s;t,x,v),V(s;t,x,v)\right)\right)\\
& +\int_0^t\cos(\omega(s-t)) \partial_{v_2}\left(fE_2\left(s,X(s;t,x,v),V(s;t,x,v)\right)\right)\\
& +\int_0^t \partial_{v_3}\left(fE_3\left(s,X(s;t,x,v),V(s;t,x,v)\right)\right)\\
& =\int_0^t \frac{\cos \sin}{\omega}\partial_{x_1}(fE_1)+\frac{\cos (\cos-1)}{\omega}\partial_{x_2}(fE_1)+\cos^2\partial_{v_1}(fE_1)-\cos\sin\partial_{v_2}(fE_1)\\
& +\int_0^t -\frac{\sin^2}{\omega}\partial_{x_1}(fE_2)+\frac{\sin (1-\cos)}{\omega}\partial_{x_2}(fE_2)-\cos\sin\partial_{v_1}(fE_2)+\sin^2\partial_{v_2}(fE_2)\\
& +\int_0^t \frac{(1-\cos) \sin}{\omega}\partial_{x_1}(fE_1)+\frac{\sin^2}{\omega}\partial_{x_2}(fE_1)+\sin^2\partial_{v_1}(fE_1)+\cos\sin\partial_{v_2}(fE_1)\\
& +\int_0^t \frac{\cos (1-\cos)}{\omega}\partial_{x_1}(fE_2)+\frac{\cos \sin}{\omega}\partial_{x_2}(fE_2)+\cos\sin\partial_{v_1}(fE_2)+\cos^2\partial_{v_2}(fE_2)\\
&+\int_{0}^{t}(s-t)\partial_{x_3}(fE_3)+\partial_{v_3}(fE_3)\\
& =\int_{0}^{t} \text{\rm div}_v(fE)(s,X(s;t,x,v),V(s;t,x,v))ds\\ 
& + \text{\rm div}_x\int_{0}^{t}(fH_t)(s,X(s;t,x,v),V(s;t,x,v))ds
\end{align*}
Where in the second to last equality, $\cos=\cos(\omega(s-t))$ (same for $\sin$) and $\partial_{x_i}(fE_i)$ is always evaluated at $\left(s,X(s;t,x,v),V(s;t,x,v)\right)$ (same for $\partial_{v_i}(fE_i)$).
Then we integrate with respect to $v$ which gives us \eqref{rep_rho}.
\end{proof}
\begin{remark}
	The expression of $H_t$ and the characteristics are coherent because $H_t \underset{\omega \rightarrow 0}{\longrightarrow} -t E$ and $(X^0,V^0)\underset{\omega \rightarrow 0}{\longrightarrow}(x-tv,v)$. These expressions obtained when $\omega \rightarrow 0$ correspond to the representation formula for $\rho$ in the unmagnetized case.
\end{remark}

\subsection{Control of the electric field with the characteristics}\label{ssec:controlelec}
Thanks to \cref{lem_rep_rho} which gives us a new representation formula for $\rho$, we can start to write the estimates to control the electric field, still following the steps from \cite{R05}. A first difficulty here is adapting the estimates to this new context. We also see the appearance of the singularities mentioned above at \cref{est_larget}, which will be a major difficulty.
\subsubsection{First estimates}

Thanks to the representation formula \eqref{rep_rho} for $\rho$, $E(t,\cdot)$ is given by
\begin{equation}
E(t,x)=-\left(\nabla K_3 \star \rho\right)(t,x)= E^0(t,x)+\tilde{E}(t,x)
\end{equation}
where $K_3$ is Green's function for the Laplacian in dimension 3 given by
\begin{equation}\label{Green}
K_3(x)=\frac{1}{4\pi}\frac{1}{\abs{x}}
\end{equation}
and
\begin{equation}
\left\{
\begin{aligned}
E^0(t,x)&=-\left(\nabla K_3 \star \rho_0\right)(t,x)\\
\tilde{E}(t,x)&=-\nabla K_3 \star \left(\text{\rm div}_x \int_0^t\int_v \left(fH_t\right)\left(s,X(s;t,x,v),V(s;t,x,v)\right)dvds\right)
\end{aligned}
\right.
\end{equation}
The first term $E^0$ is easier to control.
\begin{lemma}
We have the following estimate for $E^0$.
\begin{equation}\label{ineq:E0}
\norme{E^0(t,\cdot)}_{k+3} \leq C(k,\norme{f^{in}}_1,M_k(f^{in}))
\end{equation}
\end{lemma}
\begin{proof}
Thanks to the weak Young inequality, we can write
\begin{equation}
\norme{E^0(t,\cdot)}_{k+3}\leq \norme{\nabla K_3}_{\frac{3}{2},w} \norme{\rho_0(t,\cdot)}_p
\end{equation}
with $p=\frac{3k+9}{k+6}$. And the $\frac{3k+9}{k+6}$ norm of $\rho_0(t,\cdot)$ can in turn be controlled using lemma \ref{lem:ineq_mk}, where $k'=0, r=\frac{3k+9}{k+6}, p=\infty, q=1$ and with simple change of variables
\begin{equation*}
\norme{\rho_0(t,\cdot)}_\frac{3k+9}{k+6} \leq c \norme{f}^{\frac{l}{l+3}}_\infty \left(\iint \abs{v}^l f^{in} (X^0(t),V^0(t))dx dv\right)^\frac{3}{l+3}= C M_l(0)^\frac{3}{l+3}\\
\end{equation*}
with $\frac{l+3}{3}=\frac{3k+9}{k+6}$.

\noindent Since $k>3$, $\frac{l}{3}=\frac{2k+3}{k+6}\leq \frac{2k+k}{6}=\frac{k}{3}$. Hence $l\leq k$, and thanks to lemma \ref{lem:ineq_mk} with $p=\infty, q=1, k'=l$ we obtain $M_l(0)\leq c\norme{f^{in}}_1^{\frac{k-l}{k}}M_k(0)^\frac{l}{k}$.

\noindent This gives us a bound on $\rho_0(t,\cdot)$,
\begin{equation}\label{bound_rho_0}
\norme{\rho_0(t,\cdot)}_\frac{3k+9}{k+6} \leq \left(c\norme{f^{in}}_1^{\frac{k-l}{k}}M_k(0)^\frac{l}{k}\right)^\frac{3}{l+3}=C(k,\norme{f^{in}}_1,M_k(f^{in})).
\end{equation}
with $\frac{l+3}{3}=\frac{3k+9}{k+6}$.
\end{proof}

To estimate the second term $\tilde{E}$, we first notice that it can be written as
\begin{equation*}
\sum_{j,l=1}^{3} \partial_j \partial_l G_3 \star \int_0^t fH_t dv ds
\end{equation*}
so that we can apply the Calder\'on-Zygmund inequality (lemma \ref{cal_zy})
\begin{equation}
\norme{\tilde{E}(t,\cdot)}_{k+3} \leq \norme{ \underset{\Sigma(t,x)}{\underbrace{\int_0^t \int_v\left(fH_t\right)\left(s,X(s;t,x,v),V(s;t,x,v)\right)dvds}}}_{k+3}
\end{equation}
To simplify the expression of $\Sigma$, we consider the classical change of variables
\begin{align*}
\phi(v_1,v_2,v_3)& =
\begin{pmatrix}
v_1\cos(\omega (s-t))+v_2\sin(\omega (s-t))\\
-v_1\sin(\omega (s-t))+v_2\cos(\omega (s-t))\\
v_3
\end{pmatrix}\\
& =V(s;t,x,v)
\end{align*}
as well as the change of variable in time $\alpha(s)=t-s$, so that $\Sigma$ can now be written
\begin{equation}
\Sigma(t,x)=\int_0^t\int_vf(t-s,X^*(s,x,v),v)D(t-s,s,X^*(s,x,v))dvds
\end{equation}
with
\begin{equation}\label{express_D}
D\left(t,s,x\right)=
\begin{pmatrix}
-\frac{\sin(\omega s)}{\omega} E_1(t,x)+\frac{\cos(\omega s)-1}{\omega}E_2(t,x)\\
\frac{1-\cos(\omega s)}{\omega} E_1(t,x)-\frac{\sin(\omega s)}{\omega}E_2(t,x)\\
-sE_3(t,x)
\end{pmatrix}
\end{equation}
and
\begin{equation}
X^*(s,x,v)=
\begin{pmatrix}
x_1-\frac{v_1}{\omega}\sin(\omega s)+\frac{v_2}{\omega}(\cos(\omega s)-1))\\
x_2+\frac{v_1}{\omega}(1-\cos(\omega s))-\frac{v_2}{\omega}\sin(\omega s)\\
x_3-v_3 s
\end{pmatrix}	
\end{equation}
We first study $\sigma(s,t,x)$ defined by
\begin{equation}
\sigma(s,t,x)=\int_vf(t-s,X^*(s,x,v),v)D(t-s,s,X^*(s,x,v))dvds.
\end{equation}
\begin{lemma}\label{lem:estsig}
We have the following estimate for $\sigma$.
\begin{equation}\label{est_larget}
\norme{\sigma(s,t,\cdot)}_{k+3}\leq C  \frac{\sqrt{2}}{s}\left(\frac{\omega^2 s^2}{2(1-\cos(\omega s))} \right)^\frac{2}{3} M_k(t-s)^\frac{1}{k+3}
\end{equation}
\end{lemma}
\begin{proof}
Thanks to \cref{lem:lineq_1} we obtain
\begin{equation}\label{ineq_I}
\abs{\sigma(s,t,x)}\leq c \norme{D(t-s,s,X^*(s,x,\cdot))}_{\frac{3}{2},w} \norme{f}_\infty^\frac{2}{3}\norme{f(t-s,X^*(s,x,\cdot),\cdot)}_1^\frac{1}{3}
\end{equation}
Let's first look at the weak $\frac{3}{2}$-norm of $D(t-s,s,X^*(s,x,\cdot))$  in \eqref{ineq_I}. In the following computations $D$ (respectively $E$) and its coordinates $D_i$ (respectively $E_i$) are always evaluated at $(t-s,s,X^*(s,x,\cdot))$ (respectively $(t-s,X^*(s,x,\cdot))$) and $\cos=\cos(\omega s)$ (respectively $\sin=\sin(\omega s)$). 

By definition,
\begin{equation*}
\norme{D}^2_{\frac{3}{2},w}=\sum_{i=1}^{3} \norme{D_i}^2_{\frac{3}{2},w}
\end{equation*} 
so first we estimate $\norme{D_1}^2_{\frac{3}{2},w}$
\begin{align*}
\norme{D_1}^2_{\frac{3}{2},w}& \leq \frac{\sin^2}{\omega^2}\norme{E_1}^2_{\frac{3}{2},w}+\frac{(1-\cos)^2}{\omega^2}\norme{E_2}^2_{\frac{3}{2},w}+2\frac{\abs{\sin}\abs{(1-\cos)}}{\omega^2}\norme{E_1}_{\frac{3}{2},w}\norme{E_2}_{\frac{3}{2},w}\\
& \leq \frac{\sin^2}{\omega^2}\norme{E_1}^2_{\frac{3}{2},w}+\frac{(1-\cos)^2}{\omega^2}\norme{E_2}^2_{\frac{3}{2},w}+\frac{(1-\cos)^2}{\omega^2}\norme{E_1}^2_{\frac{3}{2},w}+\frac{\sin^2}{\omega^2}\norme{E_2}^2_{\frac{3}{2},w}\\
& = \frac{2(1-\cos)}{\omega^2}\left(\norme{E_1}^2_{\frac{3}{2},w}+\norme{E_2}^2_{\frac{3}{2},w}\right)
\end{align*}
The computations are the same for $\norme{D_2}^2_{\frac{3}{2},w}$ so that we can write
\begin{equation}
\norme{D}^2_{\frac{3}{2},w}  \leq \frac{4(1-\cos(\omega s))}{\omega^2}\left(\norme{E_1}^2_{\frac{3}{2},w}+\norme{E_2}^2_{\frac{3}{2},w}\right)+s^2\norme{E_3}^2_{\frac{3}{2},w}
\end{equation}
and since for all $x \in \R$, $2(1-\cos(x))\leq x^2$
\begin{equation}\label{nwH}
\norme{D}^2_{\frac{3}{2},w} \leq 2s^2\left(\norme{E_1}^2_{\frac{3}{2},w}+\norme{E_2}^2_{\frac{3}{2},w}\right)+s^2\norme{E_3}^2_{\frac{3}{2},w} \leq 2s^2\norme{E}^2_{\frac{3}{2},w}
\end{equation}
Now let's try to express $\norme{E_1(t-s,X^*(s,x,\cdot))}_{\frac{3}{2},w}$, by definition
\begin{equation}
\norme{E_1(t-s,X^*(s,x,\cdot))}_{\frac{3}{2},w}=\underset{\abs{A}<\infty}{\sup} \abs{A}^{-\frac{1}{3}} \int_A \abs{E_1(t-s,X^*(s,x,v))}dv
\end{equation}
and if we consider the change of variables $\psi(v)=X^*(s,x,v)$, for $s>0$, whose Jacobian matrix is given by
\begin{equation}
\text{\rm Jac}(\psi)=
\begin{pmatrix}
-\sin(\omega s) & \cos(\omega s)-1 & 0\\
1-\cos(\omega s) & -\sin(\omega s) & 0\\
0 & 0 & -s
\end{pmatrix}
\end{equation}
we can write
\begin{equation*}
\int_A \abs{E_1(t-s,X^*(s,x,v))}dv =\int_{\psi(A)} \abs{E_1(t-s,u))}\abs{\text{\rm Jac}(\psi)}^{-1}du
\end{equation*}
So finally
\begin{align*}
& \norme{E_1(t-s,X^*(s,x,\cdot))}_{\frac{3}{2},w}\\
& =\underset{\abs{A}<\infty}{\sup} \abs{A}^{-\frac{1}{3}} \int_{\psi(A)} \abs{E_1(t-s,u))}\abs{\text{\rm Jac}(\psi)}^{-1}du\\
& = \underset{\abs{A}<\infty}{\sup} \abs{\psi(A)}^{-\frac{1}{3}}\underset{=\abs{\text{\rm Jac}(\psi)}^{-1}}{\left(\underbrace{\frac{\abs{A}}{\abs{\psi(A)}}}\right)}^{-\frac{1}{3}}\abs{\text{\rm Jac}(\psi)}^{-1} \int_{\psi(A)} \abs{E_1(t-s,u))}du\\
& = \underset{\abs{A}<\infty}{\sup} \abs{\psi(A)}^{-\frac{1}{3}}\abs{\text{\rm Jac}(\psi)}^{-\frac{2}{3}} \int_{\psi(A)} \abs{E_1(t-s,u))}du\\
& = \abs{\text{\rm Jac}(\psi)}^{-\frac{2}{3}} \norme{E_1(t-s,\cdot)}_{\frac{3}{2},w}
\end{align*}
The computations are the same for $\norme{E_2(t-s,X^*(s,x,\cdot))}_{\frac{3}{2},w}$ and\\ $\norme{E_3(t-s,X^*(s,x,\cdot))}_{\frac{3}{2},w}$ so that
\begin{equation}\label{nwE}
\begin{aligned}
\norme{E(t-s,X^*(s,x,\cdot))}_{\frac{3}{2},w} & =\abs{\text{\rm Jac}(\psi)}^{-\frac{2}{3}} \norme{E(t-s,\cdot)}_{\frac{3}{2},w}\\
& =\left(\frac{1}{2s(1-\cos(\omega s))}\right)^\frac{2}{3}\norme{E(t-s,\cdot)}_{\frac{3}{2},w}
\end{aligned}
\end{equation}
Combining \eqref{nwH} and \eqref{nwE} we obtain the following estimate
\begin{equation}
\norme{D(t-s,s,X^*(s,x,\cdot))}_{\frac{3}{2},w}\leq \frac{\sqrt{2}}{s}\left(\frac{\omega^2 s^2}{2(1-\cos(\omega s))}\right)^\frac{2}{3}\underset{\leq C}{\underbrace{\norme{E(t-s,\cdot)}_{\frac{3}{2},w}}}.
\end{equation}
and since $\norme{f}_\infty \leq C$ we have
\begin{equation}
\abs{\sigma(s,t,x)} \leq C \frac{\sqrt{2}}{s}\left(\frac{\omega^2 s^2}{2(1-\cos(\omega s))}\right)^\frac{2}{3}\norme{f(t-s,X^*(s,x,\cdot),\cdot)}_1^\frac{1}{3}ds.
\end{equation}
So that
\begin{equation}
\norme{\sigma(s,t,\cdot)}_{k+3}\leq C  \frac{\sqrt{2}}{s}\left(\frac{\omega^2 s^2}{2(1-\cos(\omega s))} \right)^\frac{2}{3} \norme{\left(\int f(t-s,X^*(s,\cdot,v),v)dv\right)^\frac{1}{3}}_{k+3}
\end{equation}
Furthermore, for any function $\psi$ we have
\begin{equation}\label{norme_puissance}
\norme{\psi^\alpha}_p=\norme{\psi}^\alpha_{\alpha p}
\end{equation}
so that
\begin{equation}
\norme{\left(\int f(t-s,X^*(s,\cdot,v),v)dv\right)^\frac{1}{3}}_{k+3} \leq \norme{\int f(t-s,X^*(s,\cdot,v),v)dv}^\frac{1}{3}_\frac{k+3}{3},
\end{equation}
and thanks to \cref{lem:ineq_mk} with $p=\infty, q=1, k'=0, r=\frac{k+3}{3}$ we obtain the desired estimate
\begin{equation*}
\norme{\sigma(s,t,\cdot)}_{k+3}\leq C  \frac{\sqrt{2}}{s}\left(\frac{\omega^2 s^2}{2(1-\cos(\omega s))} \right)^\frac{2}{3} M_k(t-s)^\frac{1}{k+3}
\end{equation*}
with $C=C(k,\norme{f^{in}}_1,\norme{f^{in}}_\infty,\mathcal{E}_{in})$.
\end{proof}

Like in the unmagnetized case, we exactly obtain the desired exponent $\frac{1}{k+3}$ on $M_k$ in our estimate. However, as mentioned above, we also see the singularities at times $\frac{2\pi k}{\omega}, k \in \NN$.

To deal with the singularities that stem from the added magnetic field, we notice that all our estimates depend only on $k,\omega$ and $f^{in}$, which means that if we can show propagation of moments on an interval $\intervalleff{0}{T_\omega}$, then we can reiterate our analysis with the new initial condition $f^{in}_1=f(T_\omega)$ and so on.

Since the singularities depend on $\omega$, it is logical to take $T_\omega$ that also depends on $\omega$ (this also justifies the notation). As said above, we choose to take $T_\omega=\frac{\pi}{\omega}$ (in fact, we could have taken any $t \in \intervalleoo{0}{\frac{2\pi}{\omega}}$).

Now to control $\norme{\Sigma(t,\cdot)}_{k+3}$ with $M_k(t)^\frac{1}{k+3}$ we write
\begin{equation}\label{Sigma}
\Sigma(t,x) :=\int_{0}^{t0}...+\int_{t_0}^{t}...
\end{equation}
where $t_0 \in \intervalleoo{0}{T_\omega}$.
This is an idea from the original paper \cite{LP91}. The interval $\intervalleff{0}{t_0}$ is considered small and thus we control the large $t$ contribution ($\int_{t_0}^{t}$) precisely (with $M_k(t)^\beta$, $\beta \leq \frac{1}{k+3}$) and the small $t$ contribution ($\int_{0}^{t_0}$) less precisely (with $M_k(t)^\gamma$, $\gamma >0$). This last imprecise estimate is compensated by the fact that we integrate on a short length segment. However, the main difference with the unmagnetized case is that now we need $t_0$ to be small compared to $T_\omega=\frac{\pi}{\omega}$ to deal with the singularities.

\subsubsection{Small time estimates}
First we estimate the small contribution in time, as in \cite{R05}, but with the added difficulty of the singularities.
\begin{proposition}
We have the following estimate for the small contribution in time
\begin{equation}\label{ineq:estsmallt}
\norme{\int_0^{t_0} \sigma(s,t,\cdot)ds}_{k+3}
\leq C (\omega t_0)^{2-\frac{3}{d}}(1+t)^\frac{l+3}{k+3}\left(1+\underset{0\leq s\leq t}{\sup} M_k(s)\right)^\frac{3(l+3)}{(k+3)^2}
\end{equation}
with $C=C(k,\norme{f^{in}}_1,\norme{f^{in}}_\infty,\mathcal{E}_{in})$ and $l$ is an exponent defined in the proof.
\end{proposition}
\begin{proof}
Thanks to the H\"older inequality with $\frac{1}{d}+\frac{1}{d'}=1$, we can write
\begin{align*}
& \abs{\sigma(s,t,x)}\\
& \leq \left(\int_{\R^3}\abs{D\left(t-s,s,X^*(s,x,v)\right)}^d dv\right)^\frac{1}{d}\left(\int_{\R^3} f\left(t-s,X^*(s,x,v),v\right)^{d'}dv\right)^\frac{1}{d'}\\
& \leq \sqrt{2}s\left(\frac{1}{2s(1-\cos(\omega s))}\right)^\frac{1}{d}\norme{E(t-s,\cdot)}_d \norme{f}_\infty^\frac{1}{d} \left(\int_{\R^3} f\left(t-s,X^*(s,x,v),v\right)dv\right)^\frac{1}{d'}.
\end{align*}
Using \cref{norme_puissance} with $\alpha=\frac{1}{d'}, p=k+3$ and \cref{lem:ineq_mk} with $p=\infty, q=1, k'=0, r=\frac{k+3}{d'}$, this implies

\begin{align*}
& \norme{\int_0^{t_0} \sigma(s,t,\cdot)ds}_{k+3}\\
& \leq C \underset{0\leq s\leq t}{\sup}\norme{E(t-s,\cdot)}_d \\
& \times \underset{0\leq s\leq t}{\sup} \norme{\left(\int_{\R^3} f\left(t-s,X^*(s,x,v),v\right)dv\right)}_{\frac{k+3}{d'}}^\frac{1}{d'} \int_0^{t_0} s\left(\frac{1}{s(1-\cos(\omega s))}\right)^\frac{1}{d} ds\\
& \leq C \underset{0\leq s\leq t}{\sup}\norme{E(t-s,\cdot)}_d \underset{0\leq s\leq t}{\sup} M_l(t-s)^\frac{1}{k+3} \int_0^{t_0} s\left(\frac{1}{s(1-\cos(\omega s))}\right)^\frac{1}{d} ds
\end{align*}
where thanks to \cref{lem:ineq_mk}, the new exponent $l$ verifies $\frac{k+3}{d'}=\frac{l+3}{3}$. Furthermore, we saw in \cref{lem:estE} that the electric field is uniformly bounded in $L^d(\R^3)$ for $\frac{3}{2}< d \leq \frac{15}{4}$ (so $\frac{15}{11}\leq d' < 3$). This implies the following estimate, with $\frac{k+3}{d'}=\frac{l+3}{3}$ and $\frac{15}{11}\leq d' < 3$,
\begin{equation}
\label{est_smallt}
\norme{\int_0^{t_0} \sigma(s,t,\cdot)ds}_{k+3} \leq C \left(\int_0^{t_0} s\left(\frac{1}{s(1-\cos(\omega s))}\right)^\frac{1}{d} ds\right) \underset{0\leq s\leq t}{\sup} M_l(s)^\frac{1}{k+3}.
\end{equation}
Thanks to \cref{est_ml} we have that
\begin{equation*}
\underset{0\leq s\leq t}{\sup} M_l(s) \leq C(1+t)^{l+3}\left(1+\underset{0\leq s\leq t}{\sup} M_k(s)\right)^\frac{3(l+3)}{k+3}
\end{equation*}
so that finally we obtain
\begin{equation}\label{smallt}
\begin{aligned}
&\norme{\int_0^{t_0} \sigma(s,t,\cdot)ds}_{k+3}\\  
& \leq C \left(\int_0^{t_0} \underset{\zeta(s)}{\underbrace{s\left(\frac{1}{s(1-\cos(\omega s))}\right)^\frac{1}{d}}} ds\right)(1+t)^\frac{l+3}{k+3}\left(1+\underset{0\leq s\leq t}{\sup} M_k(s)\right)^\frac{3(l+3)}{(k+3)^2}.
\end{aligned}
\end{equation}
with $C=C(k,\norme{f^{in}}_1,\norme{f^{in}}_\infty,\mathcal{E}_{in})$.

Now we must study $\int_0^{t_0} \zeta(s) ds$ (in the case without magnetic field $I=\intervalleff{0}{t_0}$ and $\zeta(s)=s^{1-\frac{3}{d}}$).

We have
\begin{equation*}
\begin{aligned}
\int_0^{t_0} \zeta(s) ds & = \omega^{\frac{1}{d}-2}\int_0^{\omega t_0} 	s\left(\frac{1}{s(1-\cos(s))}\right)^\frac{1}{d} ds\\
& =\omega^{\frac{1}{d}-2}\int_0^{\omega t_0} 	s^{1-\frac{3}{d}}\left(\frac{s^2}{(1-\cos(s))}\right)^\frac{1}{d} ds
\end{aligned}
\end{equation*}
Since $\omega t_0 \leq \omega t \leq \pi$, the function $s \mapsto \left(\frac{s^2}{(1-\cos(s))}\right)^\frac{1}{d}$ is bounded on $\intervalleff{0}{\omega t_0}$ (independently of $t_0$) so that finally
\begin{equation}\label{smallt1}
\int_0^{t_0} \zeta(s) ds \leq C \int_0^{\omega t_0} s^{1-\frac{3}{d}}ds \leq C(\omega t_0)^{2-\frac{3}{d}}
\end{equation}
\end{proof}
\subsubsection{Large time estimates}
Now we look at the large $t$ contribution, where our hope is to get a logarithmic dependence in $t_0$ just like in \cite{LP91,R05}. 
\begin{proposition}
We have the following estimate for the large contribution in time
\begin{equation}\label{est_larget2}
\norme{\int_{t_0}^t\sigma(s,t,\cdot)ds}_{k+3} \leq C \ln \left(\frac{t}{t_0}\right)\underset{0\leq s\leq t}{\sup}  M_k(s)^\frac{1}{k+3} 
\end{equation}
with $C=C(k,\norme{f^{in}}_1,\norme{f^{in}}_\infty,\mathcal{E}_{in})$.
\end{proposition}
\begin{proof}
Using \cref{est_larget}, we can write

\begin{align*}
\norme{\int_{t_0}^t\sigma(s,t,\cdot)ds}_{k+3} & \leq C \underset{0\leq s\leq t}{\sup}  M_k(s)^\frac{1}{k+3} \int_{\omega t_0}^{\omega t}\frac{1}{s}\left(\frac{ s^2}{(1-\cos( s))} \right)^\frac{2}{3}ds \\
& \leq C \underset{0\leq s\leq t}{\sup}  M_k(s)^\frac{1}{k+3} \int_{\omega t_0}^{\omega t}\frac{1}{s}ds\\ 
\end{align*}
because in the same way as above the function $s \mapsto \left(\frac{s^2}{(1-\cos(s))}\right)^\frac{1}{d}$ is bounded on $\intervalleff{\omega t_0}{\omega t}$ (independently of $t_0,t$ or $\omega$) so that finally
\begin{equation*}
\norme{\int_{t_0}^t\sigma(s,t,\cdot)ds}_{k+3} \leq C \ln \left(\frac{t}{t_0}\right)\underset{0\leq s\leq t}{\sup}  M_k(s)^\frac{1}{k+3} 
\end{equation*}
with $C=C(k,\norme{f^{in}}_1,\norme{f^{in}}_\infty,\mathcal{E}_{in})$.
\end{proof}
\subsection{A Gr\"onwall inequality for $t \in \intervalleff{0}{T_\omega}$}
\label{ssec:gronwall}
Now we try to show propagation of moments on $\intervalleff{0}{T_\omega}$ by establishing a Gr\"onwall inequality like in \cite{LP91,R05} while 
\begin{proposition}\label{prop:Tw}
\cref{theo:main} is true for $T=T_\omega$.
\end{proposition}
\begin{proof}
First, we define
\begin{equation}\label{def_mu}
\mu_k(t):=\underset{0\leq s\leq t}{\sup}  M_k(s)
\end{equation}
Next combining \cref{bound_rho_0}, \cref{smallt}, \cref{smallt1}, and \cref{est_larget2}, we obtain the following estimate for all $t\in \intervalleff{0}{T}$
\begin{equation}
\begin{aligned}\label{ineq:finalE1}
\norme{E(t,\cdot)}_{k+3} & \leq \norme{\rho_0(t,\cdot)}_\frac{3k+9}{k+6}+C (\omega t_0)^{2-\frac{3}{d}}(1+t)^\frac{l+3}{k+3}\left(1+\mu_k(t)\right)^\frac{3(l+3)}{(k+3)^2}\\
& +C\ln\left(\frac{t}{t_0}\right)\mu_k(t)^\frac{1}{k+3}
\end{aligned}
\end{equation}
Now, as was previously announced, we can absorb the term $\left(1+\mu_k(t)\right)^\frac{3(l+3)}{(k+3)^2}$ by choosing a small $t_0$ such that $t_0<t \leq T_\omega$.
We choose $t_0$ in a different way than what was done in \cite{LP91} and \cite{R05} by using the natural variable $\frac{t}{t_0}$. Hence $t_0$ is defined by the following relation
\begin{equation}\label{choice_t0}
(\frac{t_0}{t})^{2-\frac{3}{d}}\left(1+\mu_k(t)\right)^\frac{3(l+3)}{(k+3)^2}=1
\end{equation}
(the exponent $2-\frac{3}{d}$ is non-negative). Thus, we automatically have the inequality $t_0 < t \leq T_\omega$. 

Then we can bound the three terms in \cref{ineq:finalE1} so as to obtain
\begin{equation}\label{ineq:finalE}
\begin{aligned}
\norme{E(t,\cdot)}_{k+3}& \leq C_1+ C_2 t^{2-\frac{3}{d}}(1+t)^\frac{l+3}{k+3}+C_3\frac{3(l+3)}{(2-\frac{3}{d})(k+3)^2}\mu_k(t)^\frac{1}{k+3}\ln\left(1+\mu_k(t)\right)\\
& \leq C\left(1+\mu_k(t)\right)^\frac{1}{k+3}\left(1+\ln\left(1+\mu_k(t)\right)\right)
\end{aligned}
\end{equation}
with $C=C(T,k,\omega,\norme{f^{in}}_1,\norme{f^{in}}_\infty,\mathcal{E}_{in},M_k(f^{in}))$.

So now thanks to the inequality \cref{eq:diff_Mk} we can write
\begin{equation}
\frac{d}{dt}M_k(t)\leq C\left(1+\mu_k(t)\right)\left(1+\ln\left(1+\mu_k(t)\right)\right)
\end{equation}
and integrating the inequality on $\intervalleff{0}{t}$ we conclude that
\begin{equation*}
M_k(t)\leq M_k(0)+C\int_{0}^{t}\left(1+\mu_k(s)\right)\left(1+\ln\left(1+\mu_k(s)\right)\right)ds
\end{equation*}
for all $t \in \intervalleff{0}{T}$.

\noindent
Setting $y(t)=1+\mu_k(t)$, we have
\begin{equation}
0<y(t)\leq y(0)+C\int_{0}^{t}y(s)(1+\ln y(s))ds
\end{equation}
thus
\begin{equation}
\frac{Cy(t)(1+\ln y(t))}{y(0)+C\int_{0}^{t}y(s)(1+\ln y(s))ds}\leq C(1+\ln y(t))ds
\end{equation}
and integrating in time gives
\begin{equation}
\ln\left(\frac{y(t)}{y(0)}\right)\leq\ln\left(\frac{y(0)+C\int_{0}^{t}y(s)(1+\ln y(s))ds}{y(0)}\right)\leq C\int_{0}^{t}(1+\ln y(s))ds.
\end{equation}
Hence $t\mapsto \ln y(t)$ verifies a classical Gr\"onwall inequality
\begin{equation}\label{rec}
\ln y(t) \leq \ln y(0)+ Ct+C\int_{0}^{t}\ln y(s)ds \leq \ln y(0)+ CT+C\int_{0}^{t}\ln y(s)ds
\end{equation} 
which implies
\begin{equation}
\ln y(t) \leq \left(\ln y(0)+ CT\right)\exp\left(C t\right) \Leftrightarrow y(t) \leq \exp\left(CT\exp\left(C t\right)\right) y(0)^{\exp\left(C t\right)}
\end{equation}
for all $t \in \intervalleff{0}{T}$ with $C=C(T,k,\omega,\norme{f^{in}}_1,\norme{f^{in}}_\infty,\mathcal{E}_{in},M_k(f^{in}))$. 
\end{proof}

\subsection{Propagation of moments for all time}\label{ssec:propaallt}
We conclude the proof of \cref{theo:main} by showing propagation of moments for all time.
Since the constant $C$ in our estimate in \cref{prop:Tw} depends only on $T,k,\omega,\norme{f^{in}}_1,\norme{f^{in}}_\infty,\mathcal{E}_{in}$ and $M_k(f^{in})$, we can reiterate the procedure on any time interval $I_p=\intervalleff{pT_\omega}{(p+1)T_\omega}$. Indeed, $T,k$ and $\omega$ are constant $\norme{f(t)}_1$ and $\norme{f(t)}_\infty$ are conserved in time, the energy is bounded and $M_k(f)$ is exactly the quantity we are studying.

\begin{proposition}\label{prop:allt}
\cref{theo:main} is true for all $T>T_\omega$.
\end{proposition}
\begin{proof}
	First, we show by induction on $n$ that for all $n\in \NN^*$
	\begin{equation}\label{eq:induction}
	y(nT_\omega) \leq \beta_{n-1} \beta_{n-2}^{\alpha_{n-1}}\beta_{n-3}^{\alpha_{n-1}\alpha_{n-2}}...\beta_{0}^{\alpha_{n-1}\alpha_{n-2}...\alpha_1} y(0)^{\alpha_{n-1}...\alpha_{0}}
	\end{equation}
with $\beta_p=\exp\left(C_pT\exp\left(C_p T_\omega\right)\right)$ and $\alpha_p=\exp\left(C_p T_\omega\right)$ with
\begin{align*}
C_p & =C_p(T,k,\omega,\norme{f^{in}}_1,\norme{f^{in}}_\infty,\mathcal{E}_{in},M_k(f(pT_\omega)))\\
& = C_p(T,k,\omega,\norme{f^{in}}_1,\norme{f^{in}}_\infty,\mathcal{E}_{in},M_k(f^{in})).
\end{align*}

The initial case is simply a consequence of \cref{prop:Tw}. Proving the induction step is also easy because thanks to the induction hypothesis, $f(nT_\omega)$ verifies the assumptions of \cref{theo:main}. This means we can apply the same analysis as in the previous subsections while initializing system \cref{sys:VPwB} with $f(nT_\omega)$. 

Hence we obtain:
\begin{equation}
y((n+1)T_\omega) \leq \exp\left(C_nT\exp\left(C_n t\right)\right) y(nT_\omega)^{\exp\left(C_n t\right)}
\end{equation}
with $C_n=C_n(T,k,\omega,\norme{f^{in}}_1,\norme{f^{in}}_\infty,\mathcal{E}_{in},M_k(f(nT_\omega)))$.  The induction step is completed by writing $\beta_{n}=\exp\left(C_nT\exp\left(C_n t\right)\right)$ and $\alpha_{n}={\exp\left(C_n t\right)}$ and by applying the induction hypothesis \ref{eq:induction}.

To conclude we consider $t\in \intervalleff{0}{T}$ with $T>T_\omega$ and we write $t=(n+r)T_\omega$ with $n \in \NN$ and $0\leq r < 1$.
Like in the induction proof above, we can apply the same analysis as in the previous section while initializing with $f(nT_\omega)$ to obtain:
\begin{equation}
y(t) \leq \exp\left(C_nT\exp\left(C_n \left(t-\frac{n\pi}{\omega}\right)\right)\right)y(\frac{n\pi}{\omega})^{\exp\left(C_n \left(t-\frac{n\pi}{\omega}\right)\right)}.
\end{equation}
The proof is complete since we showed just before that we can bound $y(\frac{n\pi}{\omega})$.
\end{proof}

\subsection{Difficulty of controlling the electric field with the magnetic field in the source term}\label{ssec:Bsource}
	In this section, we present a strategy for the proof of \cref{theo:main} that does not permit us to conclude, but which is still interesting to detail because of its simplicity. 
	
	The idea is to consider the magnetic term $v\wedge B \cdot \nabla_v $ not as an added transport term in the Vlasov equation but as a source term. This allows us to write a new representation formula for the macroscopic density using the characteristics of the unmagnetized Vlasov-Poisson system.
\begin{lemma}	
	We have a representation formula for $\rho$,
	\begin{equation}\label{rep_rho2}
	\rho(t,x)=\rho_0(t,x)-\text{\rm div}_x \int_0^t s\int_v \left(f\left(E+v \wedge B\right)\right)\left(t-s,x-sv,v\right)dv ds
	\end{equation}
\end{lemma}
\begin{proof}
 We use the methods of characteristics and the Duhamel formula but this time with the magnetic term in the source term, which allows us to write
\begin{align*}
f(t,x,v)&=f^{in}(x-tv,v)\\
& -\int_{0}^{t}\left(E+v\wedge B\right)(s,x+(s-t)v)\cdot \nabla_v f(s,x+(s-t)v,v)ds\\
&=f^{in}(x-tv,v)-\int_{0}^{t} \text{\rm div}_v\left( \left(E+v\wedge B\right)f\right)(t-s,x-sv,v)ds
\end{align*}
where we used the change of variable $s=t-s$ and because $\text{\rm div}_v \left(E+v\wedge B\right)=0$.

\noindent Now we notice that
\begin{align*}
\text{\rm div}_v\left( \left(E+v\wedge B\right)f(t-s,x-sv,v)\right)&=-s\text{\rm div}_x\left( \left(E+v\wedge B\right)f(t-s,x-sv,v)\right)\\
&+\text{\rm div}_v\left( \left(E+v\wedge B\right)f\right)(t-s,x-sv,v)
\end{align*}
Using this equality and integrating in $v$ we obtain \eqref{rep_rho2}.
\end{proof}

\noindent Now we define
\begin{equation}\label{sigmas}
\left\{
\begin{aligned}
&\Sigma_E(t,x)=\int_0^t s \int_v E(t-s,x-sv)f\left(t-s,x-sv,v\right)dvds \\
&\Sigma_B(t,x)=\int_0^t s \int_v v\wedge B(t-s,x-sv)f\left(t-s,x-sv,v\right)dvds\\
&\Sigma(t,x)= \Sigma_E(t,x)+\Sigma_B(t,x)
\end{aligned}
\right.
\end{equation}
Thanks to the Calder\'on-Zygmund inequality, to estimate the $k+3$-norm of $E(t,\cdot)$, we only need to estimate the $k+3$-norms of $\Sigma_E(t,\cdot)$ and $\Sigma_B(t,\cdot)$.

Using the exact same analysis as in \cite{LP91,R05}, we obtain the following estimate for $\Sigma_E(t,\cdot)$ with $\mu(t)$ defined as in \eqref{def_mu}
\begin{equation}
\norme{\Sigma_E(t,\cdot)}_{k+3} \leq Ct_0^{2-\frac{3}{d}}(1+t)^\frac{l+3}{k+3}\left(1+\mu_k(t)\right)^\frac{3(l+3)}{(k+3)^2}\\
+C\ln\left(\frac{ t}{ t_0}\right)\mu_k(t)^\frac{1}{k+3},
\end{equation}
and then we choose $t_0$ like at \eqref{choice_t0} to obtain
\begin{equation}
\norme{\Sigma_E(t,\cdot)}_{k+3} \leq C\left(1+\mu_k(t)\right)^\frac{1}{k+3}\left(1+\ln\left(1+\mu_k(t)\right)\right)
\end{equation}
which is a good estimate, analogous to \cref{ineq:finalE}.

Next we try to estimate $\norme{\Sigma_B(t,\cdot)}_{k+3}$
\begin{align*}
\abs{\Sigma_B(t,x)} &=\omega \abs{\int_{0}^{t} s \int_v \begin{pmatrix}
	v_2\\
	-v_1\\
	0
	\end{pmatrix}f\left(t-s,x-sv,v\right)dv ds}\\
&\leq \omega \int_{0}^{t} s \int_v \abs{v}f\left(t-s,x-sv,v\right)dv ds=\omega \int_{0}^{t} s m_1(f\left(t-s,x-s\cdot,\cdot\right))ds
\end{align*}
So that 
\begin{equation*}
\begin{aligned}
\norme{\Sigma_B(t,\cdot)}_{k+3} & \leq \omega\int_{0}^{t}{s}ds \underset{0\leq s\leq t}{\sup}\norme{m_1(f\left(t-s,x-s\cdot,\cdot\right))}_{k+3}\\
& =\omega t^2 \underset{0\leq s\leq t}{\sup}\norme{m_1(f\left(t-s,x-s\cdot,\cdot\right))}_{k+3}
\end{aligned}
\end{equation*}
Unfortunately, $\norme{m_1(t)}_{k+3}$ can't be controlled by $M_k(t)^\alpha$ because when we apply lemma \ref{lem:ineq_mk} with $p=\infty,q=1,k'=1$ (which is the optimal case) we obtain
\begin{equation}
\norme{m_1(t)}_{k+3} \leq c \norme{f}^\frac{l-1}{l+3}_\infty M_l(t)^\frac{4}{l+3}
\end{equation}
with $k+3=\frac{l+3}{4}$ which implies $l>k$.

Indeed, its seems logical that with the added $v$ in the magnetic part of the Lorentz force, controlling $\Sigma_B$ requires a velocity moment of higher order than with $\Sigma_E$.
Thus $\norme{\Sigma_B(t,\cdot)}_{k+3}$ can't be controlled with $M_k(t)$, which means we can't deduce a Gr\"onwall inequality on $M_k(t)$ with this method.
\section{Proof of additional results}\label{sec:proofadditional}

\subsection{Proof of propagation of regularity}\label{ssec:proofreg}
First we begin by presenting the proof of the propagation of regularity. Here we directly adapt subsection 4.5 of \cite{G13}. We only present in detail the parts of the proof that involve the added magnetic field.

\begin{remark}
	The mass conservation and the energy bound can be directly deduced from the assumptions of \ref{theo:reg}
	\begin{equation}
	\iint_{\R^3\times \R^3} f(t,x,v)dxdv = \mathcal{M}^{in}=\iint_{\R^3\times \R^3}f^{in}dxdv < \infty
	\end{equation}
	\begin{equation}
	\frac{1}{2}\iint_{\R^3\times \R^3}\abs{v}^2f(t,x,v)dxdv+\frac{1}{2}\int_{\R^3}\abs{E(t,x)}^2dx\leq\mathcal{E}_{in}< \infty
	\end{equation}
	for a.e. $t\geq 0$.
\end{remark}

\begin{proof}
\textit{- First step: $L^\infty$ bound for $E$}

This step is the same in both magnetized and unmagnetized cases.
We have the following bound on $E$
\begin{equation}
\norme{E(t)}_\infty \leq C_1 C_T+C_2 \mathcal{M}^{in}.
\end{equation}

\noindent
\textit{- Second step: $L^\infty$ bound for $\rho$}

\noindent
We seek to show an inequality of the type
\begin{equation}
f(t,x,v)\leq h(\abs{v}-A_T t)
\end{equation}
for all $t\in \intervalleff{0}{T}$.\\
And so we compute
\begin{equation}
\frac{d}{dt}\iint_{\R^3 \times \R^3}(f(t,x,v) - h(\abs{v}-A_T t))_+dxdv.
\end{equation}
First we can write 
\begin{align*}
&\partial_t (f(t,x,v) - h(\abs{v}-A_T t))_+ \\
& =(\partial_t f +w'(\abs{v}-A_T t)A_T)\mathds{1}_{f(t,x,v)\geq h(\abs{v}-A_T t)}\\
& =(-v\cdot \nabla_x f -(E+v \wedge B)\cdot \nabla_v f+w'(\abs{v}-A_T t)A_T)\mathds{1}_{f(t,x,v)\geq h(\abs{v}-A_T t)}\\
& = -v\cdot \nabla_{x}(f(t,x,v) - h(\abs{v}-A_T t))_+\\
& - (E+v\wedge B)\cdot \nabla_v(f(t,x,v)-h(\abs{v}-A_T t))_+\\
& + w'(\abs{v}-A_T t)\left(A_T -\underset{=E\cdot \frac{v}{\abs{v}}}{\underbrace{(E+v\wedge B)\cdot \frac{v}{\abs{v}}}}\right)\mathds{1}_{f(t,x,v)\geq h(\abs{v}-A_T t)}
\end{align*}

so finally we obtain
\begin{equation}
\begin{aligned}
& \frac{d}{dt}\iint_{\R^3 \times \R^3}(f(t,x,v) - h(\abs{v}-A_T t))_+dxdv\\
& =\iint_{\R^3 \times \R^3}w'(\abs{v}-A_T t)\left(A_T-E\cdot \frac{v}{\abs{v}}\right)\mathds{1}_{f(t,x,v)\geq h(\abs{v}-A_T t)}dxdv.
\end{aligned}
\end{equation}
We now choose $A_T=\norme{E}_\infty \Rightarrow A_T-E(t,x)\cdot \frac{v}{\abs{v}} \leq 0$ a.e., and since $w'\leq 0$ then we have
\begin{equation}
\frac{d}{dt}\iint_{\R^3 \times \R^3}(f(t,x,v) - h(\abs{v}-A_T t))_+dxdv \leq 0.
\end{equation}
So the condition
\begin{equation}
f^{in}(x,v)\leq h(\abs{v})
\end{equation}
implies that
\begin{equation}
f(t,x,v)\leq h(\abs{v}-A_T t).
\end{equation}
Since $w$ in non-increasing, this gives us the $L^\infty$ bound on $\rho$
\begin{equation}
\norme{\rho}_{L^\infty(\intervalleff{0}{T}\times \R^3)} \leq R_T
\end{equation}

\noindent
- \textit{Third step: Bound for $D_{x,v} f$}\\
We set
\begin{equation}
L(t):= \norme{D_x f(t)}_\infty + \norme{D_v f(t)}_\infty,
\end{equation}
and differentiate the Vlasov equation in $x$ and $v$ to obtain
\begin{align*}
(\partial_t +v\cdot \nabla_{x} +(E+v\wedge B)\cdot \nabla_{v})\begin{pmatrix}
D_x f\\
D_v f
\end{pmatrix}=\begin{pmatrix}
0 & D_x E(t,x)^T\\
I & D_v(v\wedge B(t,x))	
\end{pmatrix}\begin{pmatrix}
D_x f\\
D_v f
\end{pmatrix}
\end{align*}
with 
\begin{equation}
D_v(v\wedge B(t,x))=\begin{pmatrix}
0 & -B_3(t,x) & B_2(t,x)\\
B_3(t,x) & 0 & -B_1(t,x)\\
-B_2(t,x) & B_1(t,x) & 0 
\end{pmatrix}=:A(t,x)
\end{equation}
so that
\begin{equation}
(\partial_t +v\cdot \nabla_{x} +(E+v\wedge B)\cdot \nabla_{v})(\abs{D_x f}+\abs{D_v f} )\leq (1+\abs{D_x E(t,x)}+\abs{A(t,x)})(\abs{D_x f}+\abs{D_v f} ).
\end{equation}
Then setting
\begin{equation}
J(t):= \int_{0}^{t} (1+\norme{D_x E(s)}_\infty+\norme{A(s)}_\infty)ds
\end{equation}
we have
\begin{align*}
& (\partial_t +v\cdot \nabla_{x} +(E+v\wedge B)\cdot \nabla_{v})\left((\abs{D_x f}+\abs{D_v f})e^{-J(t)} \right)\\
& \leq (\abs{D_x f}+\abs{D_v f})e^{-J(t)}(\abs{D_x E(t,x)}+\abs{A(t,x)}-\norme{D_x E(t)}_\infty-\norme{A(t)}_\infty)\leq 0
\end{align*}
By the maximum principle we thus have
\begin{equation}
(\abs{D_x f}+\abs{D_v f})e^{-J(t)}\leq (\norme{D_x f(0)}_\infty+\abs{D_v f(0)})e^{-J(0)}=L(0)
\end{equation}
and finally
\begin{equation}
L(t)\leq L(0)e^{J(t)}.
\end{equation}
- \textit{Fourth step: Bound for $D_x E$}

Like in the unmagnetized case, thanks to an extension of the Calder\'on-Zygmund inequality, we can bound $D_x E(t)$
\begin{equation}
\norme{D_x E(t)}_\infty \leq C\left(1+\ln\left(1+\norme{D_x \rho(t)}_\infty\right)\right).
\end{equation}
- \textit{Fifth step: Bound for $D_x \rho$}

Like in the unmagnetized case, we can show the following bound 
\begin{equation}
\abs{D_x \rho(t,x)} \leq  R_T e^{J(t)}
\end{equation}
for all $t \in \intervalleff{0}{T}$ and a.e. $x \in \R^3$.

\noindent
\textit{- Sixth step: Last estimate}\\
Firstly, let's mention that $A \in L^\infty(\intervalleff{0}{T} \times \R^3)$ because $B \in L^\infty(\intervalleff{0}{T} \times \R^3)$.
\begin{align*}
J(t) & = \int_{0}^{t} (1+\norme{D_x E(s)}_\infty+\norme{A(s)}_\infty)ds\\
& \leq T+ \int_{0}^{t} C\left(1+\ln\left(1+\norme{D_x \rho(s)}_\infty\right)\right)+\norme{A(s)}_\infty ds\\
& \leq T(1+C+\norme{A}_\infty)+\int_{0}^{t} C\underset{\leq \ln\left((1+R_T) e^{J(s)}\right)}{\underbrace{\ln\left(1+R_T e^{J(s)}\right)}}ds\\
& \leq T(1+C+\norme{A}_\infty) + C_T T \ln(1+R_T)+C_T \int_{0}^{t} J(s)ds.
\end{align*}
Thanks to the Gr\"onwall inequality
\begin{equation}
J(t) \leq T(1+C+\norme{A}_\infty+ C_T  \ln(1+R_T))e^{TC_T}.
\end{equation}
Thus we obtain the three following estimates
\begin{equation}
\norme{D_x \rho (t)}_\infty \leq R_T \exp(T(1+C+\norme{A}_\infty+ C_T  \ln(1+R_T))e^{TC_T})=R_T',
\end{equation}

\begin{equation}
\norme{D_x E (t)}_\infty \leq C_T (1+\ln(1+R_T')).
\end{equation}
and
\begin{equation}
L(t)\leq L(0)\exp(T(1+C+\norme{A}_\infty+ C_T  \ln(1+R_T))e^{TC_T}).
\end{equation}
\end{proof}

\subsection{Proofs regarding uniqueness}\label{ssec:proofuniq}
Now we turn to the proof of \cref{theo:uniq}, which is a direct adaptation of Loeper's paper \cite{L06}.

\begin{proof}
To prove our theorem, we only need to adapt subsection 3.2 from \cite{L06}. Thus we consider two solutions of \ref{sys:VPwB} $f_1, f_2$ with initial datum $f_0$. We write the corresponding densities, electric fields and characteristics $\rho_1,\rho_2$, $E_1,E_2$ and $Y_1, Y_2$. We define the following quantity $Q$
\begin{equation}
Q(t)=\frac{1}{2} \iint_{\R^3 \times \R^3} f_0(x,\xi)\abs{Y_1(t,x,\xi)-Y_2(t,x,\xi)}^2 dxd\xi
\end{equation}
Now we only need to differentiate $Q$
\begin{align*}
 \Dot{Q}(t) & =\iint_{\R^3 \times \R^3} f_0(x,\xi)(Y_1(t,x,\xi)-Y_2(t,x,\xi)) \cdot \partial_t(Y_1(t,x,\xi)-Y_2(t,x,\xi))dxd\xi\\
& = \iint_{\R^3 \times \R^3} f_0(x,\xi)(X_1(t,x,\xi)-X_2(t,x,\xi)) \cdot  (\Xi_1(t,x,\xi)-\Xi_2(t,x,\xi))dxd\xi\\
& + \iint_{\R^3 \times \R^3} f_0(x,\xi)(\Xi_1(t,x,\xi)-\Xi_2(t,x,\xi)) \cdot ( E_1(t,X_1)-E_2(t,X_2) )dxd\xi\\
& + \iint_{\R^3 \times \R^3} f_0(x,\xi)(\Xi_1(t,x,\xi)-\Xi_2(t,x,\xi)) \cdot ((\Xi_1(t,x,\xi)-\Xi_2(t,x,\xi)) \wedge B) dxd\xi
\end{align*}
We notice that the last term (4th line) is bounded by $Q(t)$ (using the Cauchy-Schwartz inequality).
Using the analysis from \cite{L06}, we conclude that
\begin{equation}
\frac{d}{dt} Q(t) \leq C Q(t)\left(1+\ln \frac{1}{Q(t)}\right)
\end{equation}
and thus $Q(0)=0 \Rightarrow Q(t)=0$ for all $t \geq 0$.
\end{proof}
Lastly, we detail the proof of \cref{prop:boundrho}.
\begin{proof}
Like in Corollary 3 of \cite{LP91}, with $k_0 >6$, we have sufficient regularity on $E$ to consider the weak characteristics associated to system \cref{sys:VPwB}. Hence the solution to \cref{sys:VPwB} is given by
\begin{equation*}
f(t,x,v)=f^{in}(X^0(t),V^0(t))
\end{equation*}
where $X^0(t),V^0(t)=(X(0;t,x,v),V(0;t,x,v))$ and we have
\begin{equation*}
\dot{X}(s;t,x,v)=V(s;t,x,v) \quad \dot{V}(s;t,x,v)= E(s,X(s;t,x,v))+V(s;t,x,v) \wedge B
\end{equation*}
with $X(t;t,x,v)=x$ and $V(t;t,x,v)=v$. To simplify things, we write $X(s)$ and $V(s)$ for the characteristics. Since $k_0 > 6$, we can show that $E$ is bounded on $\intervalleff{0}{T}\times \R^3$ so that we can write for $s \in \intervalleff{0}{t}$ (using the same notations as in \cite{LP91})
\begin{align*}
\abs{v-V(s)}& \leq R(t-s)+\omega \int_{s}^{t}\abs{V(u)}du\\
& \leq R(t-s)+\omega \int_{s}^{t}\abs{V(u)-v}du+\omega \int_{s}^{t} \abs{v}du\\
& \leq (R+\omega \abs{v})(t-s)\exp((t-s)\omega)\\
& \leq (R+\omega \abs{v})t\exp(t\omega)
\end{align*}
where the inequality between lines 2 and 3 is obtained thanks to the basic Gr\"onwall inequality.
Hence we can now write
\begin{equation*}
\abs{x+vt-X(0)} \leq (R+\omega \abs{v})t^2 \exp(t\omega)
\end{equation*}
so that we obtain
\begin{equation}
f(t,x,v) \leq \sup \{ f^{in}(y+vt,w), \abs{y-x}\leq (R+\omega \abs{v})t^2e^{\omega t},\abs{w-v}\leq (R+\omega \abs{v})te^{\omega t}  \}
\end{equation}
The condition \cref{eq:boundrho} is deduced from this inequality in the same way as in \cite{LP91} and implies that $\rho$ is bounded.
\end{proof}
\section{Appendix} 
We present a technical estimate on the moments that we separate from the main proof to lighten the presentation, but also because the proofs are identical in both magnetized and unmagnetized cases. One can find the proof of this lemma in \cite{R05} (pages 43-44). To clarify our work, we present a more detailed version of the proof below. 

\begin{lemma}\label{est_ml}
	Let $k> 3$ and $d'\in \intervalleof{\frac{3}{2}}{\frac{15}{4}}$, then for $l$ such that $\frac{k+3}{d'}=\frac{l+3}{3}$ we have the following estimate on $M_l(t)$
	\begin{equation}\label{eq_est_ml}
	\underset{0\leq s\leq t}{\sup} M_l(s) \leq C(1+t)^{l+3}\left(1+\underset{0\leq s\leq t}{\sup} M_k(s)\right)^\frac{3(l+3)}{k+3}
	\end{equation}
	with $C=C(k,\norme{f^{in}}_1,\norme{f^{in}}_\infty,\mathcal{E}_{in})$.
\end{lemma}
\begin{proof}
	We first use the differential inequality \cref{eq:diff_Mk}
	\begin{equation*}
	\frac{d}{dt} M_l(t) \leq C \norme{E(t)}_{l+3}M_l(t)^{\frac{k+2}{k+3}}
	\end{equation*}
	so
	\begin{equation*}
	(l+3)\frac{d}{dt} M_l(t)^\frac{1}{l+3} \leq C \norme{E(t)}_{l+3}
	\end{equation*}
	which implies
	\begin{align*}
	M_l(t) &\leq \left(M_l(0)^\frac{1}{k+3}+\frac{C}{m+3}\int_{0}^{t}\norme{E(s,\cdot)}_{l+3}ds\right)^{m+3}\\
	& \leq \left(M_l(0)^\frac{1}{k+3}+\frac{Ct}{m+3}\underset{0\leq s\leq t}{\sup}\norme{E(s,\cdot)}_{l+3}\right)^{m+3}
	\end{align*}
	This last inequality indicates that we need to control the $q$-norm of $E(t,\cdot)$ for any $q \geq l+3$ with $M_k(t)$, and this can be done by simply using the weak Young inequality and lemma \ref{lem:ineq_mk} with $p=\infty, q=1, k'=0, r=\frac{k+3}{3}$
	
	\begin{equation}
	\norme{E(s,\cdot)}_{q}=\norme{\nabla K_3 \star \rho(t,\cdot)}_{q}\leq C\norme{\rho(t,\cdot)}_{\frac{k+3}{3}}\leq C M_k(t)^\frac{3}{k+3}
	\end{equation}
	with $1+\frac{1}{q}=\frac{2}{3}+\frac{3}{k+3} \Rightarrow q=\frac{3k+9}{6-k}$ which implies that $k<6$.
	Furthermore, we want $q\geq l+3 \Leftrightarrow 6-k\leq d' \in \intervallefo{\frac{15}{11}}{3}$ so this implies $k>3$.
	
	\noindent
	Finally, with $3<k<6$, we can choose $d'\in  \intervallefo{\frac{15}{11}}{3}$ so that $l$ defined by $\frac{k+3}{d'}=\frac{l+3}{3}$ verifies $q\geq l+3$ ($q=\frac{3k+9}{6-k}$). 
	With all this, the interpolation inequalities on $L^p$ spaces allow us to write
	\begin{equation}\label{interp}
	\norme{E(s,\cdot)}_{l+3}\leq \norme{E(s,\cdot)}^\theta_{2}\norme{E(s,\cdot)}^\theta_{q}
	\end{equation}
	$\theta \in \intervalleff{0}{1}$.
	
	\noindent
	Using this estimate and Young's classical inequality implies \eqref{eq_est_ml}.
	
	\noindent If $k\geq 6$ then for all $q\in \intervalleoo{6}{+\infty}$ there exists $3<\bar{k}<6$ such that
	\begin{equation}
	q=\frac{3\bar{k}+9}{6-\bar{k}}\quad \text{\rm and} \quad \norme{E(s,\cdot)}_{q}\leq C\norme{\rho(t,\cdot)}_{\frac{\bar{k}+3}{3}}\leq C M_{\bar{k}}(t)^\frac{3}{\bar{k}+3}.
	\end{equation}
	$M_{\bar{k}}(t)<\infty$ because thanks to lemma \ref{lem:ineq_mk} with $p=1, q=\infty, k'=\bar{k}, r=\frac{k+\frac{3}{q}}{k'+\frac{3}{q}+\frac{k-k'}{p}}=1$ we have
	\begin{equation*}
	\norme{m_{\bar{k}}(t)}_r=M_{\bar{k}}(t)\leq c \norme{f}_1^\frac{k-\bar{k}}{k}M_k(t)^\frac{\bar{k}}{k}\leq C M_k(t)^\frac{\bar{k}}{k}<\infty
	\end{equation*}
	for all $3<\bar{k}<6$.
	
	\noindent
	Thus we choose $3<\bar{k}<6$ such that $q\geq l+3$ with $\frac{k+3}{d'}=\frac{l+3}{3}$ same as before, and now we try to estimate $\norme{E(s,\cdot)}_{q}$ with $M_k(t)$
	\begin{align*}
	\norme{E(s,\cdot)}_{q}\leq C \norme{\rho(s,\cdot)}_\frac{\bar{k}+3}{3}\leq C \norme{\rho(s,\cdot)}^{1-\alpha}_1 \norme{\rho(s,\cdot)}^\alpha_\frac{k+3}{3} &\leq C\left(1+\norme{\rho(s,\cdot)}_\frac{k+3}{3}\right)\\
	& \leq C(1+M_k(t))^\frac{3}{k+3}
	\end{align*}
	This last estimate combined with the interpolation inequality \eqref{interp} results in \eqref{eq_est_ml}.
\end{proof}

\section*{Acknowledgments}
I would like to give special thanks to my advisors Bruno Despr\'es and Fr\'ed\'erique Charles at LJLL for all our discussions as well as their advice. I would also like to thank Fran\c cois Golse for his precious notes, and for helping me understand a lot of technical steps.

\bibliographystyle{siamplain}
\bibliography{referencesVPwB}
\end{document}